\documentclass[11pt]{amsart}
\usepackage{amssymb,amsmath,mathtools}
\usepackage{enumerate}
\usepackage{xcolor}
\usepackage{caption}
\usepackage[centering]{geometry}
\usepackage{url}
\usepackage{hyperref}

\usepackage{graphicx}
\usepackage{color}
\usepackage{amsmath}
\usepackage{amsfonts}
\usepackage{amssymb}
\usepackage{amscd}
\usepackage{framed}
\usepackage{bbm}
\usepackage{tcolorbox}

\usepackage{amsthm}

\newcommand{\R}{\mathbb{R}}
\newcommand{\inr}[1]{\left\langle #1 \right\rangle}

\newcommand{\PP}{\mathbb{P}}

\newtheorem{theorem}{Theorem}[section]

\theoremstyle{definition}
\newtheorem{definition}[theorem]{Definition}
\newtheorem{assumption}[theorem]{Assumption}
\newtheorem{example}[theorem]{Example}
\newtheorem{remark}[theorem]{Remark}

\newtheorem{question}[theorem]{Question}

\newcommand{\E}{\mathbb{E}}
\renewcommand{\P}{\mathbb{P}}

\numberwithin{equation}{section}

\newcommand{\eps}{\varepsilon}

\begin{document}

\title{Uniform mean estimation via generic chaining}

\author{Daniel Bartl}
\address{Department of Mathematics, Department of Statistics and Data Science, National University of Singapore}
\email{bartld@nus.edu.sg}
\author{Shahar Mendelson}
\address{Department of Mathematics,  Texas A\&M University}
\email{shahar.mendelson@gmail.com}
\date{\today}

\begin{abstract}
We introduce an empirical functional $\Psi$ that is an optimal uniform mean estimator: Let $F\subset L_2(\mu)$ be a class of mean zero functions, $u$ is a real valued function, and $X_1,\dots,X_N$ are independent, distributed according to $\mu$. We show that under minimal assumptions, with $\mu^{\otimes N}$ exponentially high probability, 
\[ \sup_{f\in F} |\Psi(X_1,\dots,X_N,f) - \mathbb{E} u(f(X))|
\leq c R(F) \frac{ \mathbb{E} \sup_{f\in F } |G_f| }{\sqrt N}, \]
where $(G_f)_{f\in F}$ is the Gaussian processes indexed by $F$ and $R(F)$ is an appropriate notion of `diameter' of the class $\{u(f(X)) : f\in F\}$.

The fact that such a bound is possible is surprising, and it leads to the solution of various key problems in high dimensional probability and high dimensional statistics.
The construction is based on combining Talagrand's generic chaining mechanism with optimal mean estimation procedures for a single real-valued random variable.
\end{abstract}

\maketitle
\setcounter{equation}{0}
\setcounter{tocdepth}{1}

\maketitle
\section{Introduction}
Empirical processes theory was developed as an attempt to obtain uniform versions of the fundamental limit laws of probability theory --- leading to the uniform law of large numbers; the uniform central limit theorem; and the uniform law of the iterated logarithm (see the books \cite{dudley2014uniform,geer2000empirical,gine2021mathematical,pollard1990empirical,vaart1996weak} for detailed surveys on these topics). 
However, over the last 25 years the focus of the theory has shifted --- because it has become apparent that quantitative estimates and not the limit behaviour of empirical processes are of central importance in Data Science. 
As a result, most of the effort in the recent study of empirical processes theory has been devoted to the following question: 
\begin{question} \label{qu:empirical-intro}
Let $(\Omega,\mu)$ be a probability space and consider $F \subset L_2(\mu)$, consisting of mean-zero functions. Set $u:\R \to \R$ that satisfies $u(0)=0$ and let 
$$
u(F) = \left\{u\circ f : f \in F\right\}.
$$ 
Let $X,X_1,...,X_N$ be independent, distributed according to $\mu$. Find sharp (high probability and in expectation) bounds on  
\begin{equation} \label{eq:empirical-intro-0}
\sup_{f \in F} \left| \frac{1}{N}\sum_{i=1}^N u\left(f(X_i)\right) - \E u(f(X)) \right| \equiv \|\PP_N-\PP\|_{u(F)}.
\end{equation}
\end{question}
Specifically, the main interest in an estimate on \eqref{eq:empirical-intro-0} is on the way the supremum $\|\PP_N-\PP\|_{u(F)}$ depends on the geometric structure of $F$---rather than on that of the class $u(F)$---, and on $N$. 

The case that received the most attention (other than the obvious choice of $u(t)=t$ which corresponds to the empirical process indexed by the class $F$ itself), is $u(t)=t^2$. 
The \emph{quadratic empirical process} appears naturally in central questions in Probability, Asymptotic Geometric Analysis and Statistics.
 In some of those applications the empirical mean $\frac{1}{N}\sum_{i=1}^N u\left(f(X_i)\right)$ is the object of interest---for example, when $X$ is an isotropic random vector in $\R^d$ (i.e., $X$ is centred and ${\rm Cov}(X)={\rm Id}$), and $F=\{\inr{\cdot,z} : z \in S^{d-1}\}$; in that case,  
$$
\|\PP_N - \PP\|_{u(F)} = \sup_{z \in S^{d-1}} \left| \frac{1}{N}\sum_{i=1}^N \inr{X_i,z}^2- 1 \right|
$$
captures the behaviour of the extremal singular values of the random matrix  whose rows are $X_i/\sqrt{N}$. 

Although using the empirical mean $\PP_N u(f)$ as an `educated guess' of the mean $\E u(f)$ is an obvious choice of a functional, by now it is well understood that it is a poor one. At the same time, there is a variety of applications where the goal is not the study of the empirical mean, but rather to obtain a uniform estimates on the means $\{\E u(f(X)) : f \in F\}$; in Sections \ref{sec:application.Lp} and \ref{sec:corruption} we present examples of that flavour that have been studied extensively in recent years. 

Over the years, several attempts have been made to construct uniform mean estimators for arbitrary classes of functions that outperform the empirical mean (notable constructions can be found in \cite{abdalla2022covariance,catoni2012challenging,lugosi2019sub,lugosi2024multivariate}). 
And while these estimators do outperform the empirical mean,  they either rely on strong structural assumptions on $F$ and exploit those heavily, or attempt to handle general classes $F$ but at the expense of suboptimal performance.

\vspace{0.5em}

To be more accurate, a \emph{mean estimation procedure} is a real-valued functional $\Psi$  that receives as input an $N$-tuple $X_1,\dots,X_N \in \Omega$, a function $f$ and a parameter $\delta$. 
The idea is that if $X_1,\dots,X_N$ are selected according to $\mu^{\otimes N}$, then with probability at least $1-\delta$, $\left| \Psi_\delta(X_1,\dots,X_N,f) - \E u(f) \right|$ is `small'; specifically, there is an event of $\mu^{\otimes N}$-probability at least $1-\delta$ on which one has sufficient control on 
\[
\sup_{f \in F} \left| \Psi_\delta(X_1,\dots,X_N,f) - \E u(f) \right|.
\]

The main goal of this note is to construct an \emph{optimal} uniform mean estimator for an arbitrary class $F$ and a general $u$ under minimal assumptions. As a result, our focus is on the following alternative to Question \ref{qu:empirical-intro}.

\begin{tcolorbox}
\begin{question} \label{qu:main-intro}
Let $X_1,...,X_N$ be independent, distributed according to $\mu$. Given $\eps>0$ and $0<\delta<1$, is there a functional $\Psi$ that receives as data the points $X_1,...,X_N$ and the wanted confidence parameter $\delta$, and on an event with probability at least $1-\delta$ with respect to $\mu^{\otimes N}$, returns $\Psi_{\delta}(X_1,...,X_N,f)$ for which 
\begin{align}
\label{eq:error.intro}
\sup_{f \in F} \left| \Psi_{\delta}(X_1,...,X_N,f) - \E u(f) \right| \leq \eps?
\end{align}
\end{question}
\end{tcolorbox}

Setting our aim high, one may consider a wildly optimistic conjecture regarding a possible  answer to Question \ref{qu:main-intro}: that the error in \eqref{eq:error.intro} should scale as 
\begin{equation} \label{eq:nice-perform}
\sim {\rm diam}(u(F)) \frac{\E\sup_{f \in F}G_f}{\sqrt{N}},
\end{equation}
where $\{G_f : f \in F\}$ is the centred Gaussian process indexed by $F \subset L_2(\mu)$ whose covariance is endowed by $L_2(\mu)$\footnote{That is,  ${\rm  Cov}(G_f,G_h) = \|f-h\|_{L_2}$ for every $f,h\in F$.}, and ${\rm diam}(u(F))$ is some natural notion of the diameter of the class $u(F)$. 

The estimate \eqref{eq:nice-perform} is, in fact, true when $u(t)=t^2$ and the class consists of mean zero functions and satisfies a subgaussian condition: namely, that for every $f,h \in F \cup \{0\}$ and every $q \geq 2$,
\begin{equation} \label{eq:subgaussian}
\|f-h\|_{L_q} \leq L\sqrt{q}\|f-h\|_{L_2}.
\end{equation}
Setting $d_F=\sup_{f \in F} \|f\|_{L_2}$, one can show that for a subgaussian, centrally symmetric\footnote{That is, if $f \in F$ then $-f \in F$.} class $F$, if $N \geq c_0(\frac{ \E\sup_{f \in F}G_f }{d_F})^2$, then
\begin{equation} \label{eq:subgaussian-est-2}
\E \sup_{f \in F} \left|\frac{1}{N} \sum_{i=1}^N f^2(X_i) - \E f^2 \right| 
\leq c_1(L) d_F \frac{\E \sup_{f \in F} G_f}{\sqrt{N}},
\end{equation}
and that \eqref{eq:subgaussian-est-2} is the best estimate one can hope for.
The proof of \eqref{eq:subgaussian-est-2} and other results on empirical processes involving subgaussian classes can be found in \cite{mendelson2016upper}. 

Unfortunately, even for $u(t)=|t|^p$, this phenomenon---that the empirical mean $\PP_N u(f)$ has the error \eqref{eq:nice-perform}---, does not go beyond the case $p=2$ in the sense that if $u(t)=|t|^p$ for some $p>2$ (even very close to 2), then $ \sup_{f \in F} |\frac{1}{N} \sum_{i=1}^N u(f(X_i)) - \E u(f) |$ can be much bigger than \eqref{eq:nice-perform}.

Indeed, let $X=(g_1,...,g_d)$ be the standard Gaussian random vector in $\R^d$, put $F=\{\inr{\cdot, z} : z \in S^{d-1}\}$,  and denote the Euclidean norm by $\|\cdot\|_2$.
Note that $\{G_f : f\in F\} = \{\inr{X,z} : z\in S^{d-1}\}$, that $\sup_{z\in S^{d-1}} \inr{X,z} = \|X\|_2$, and that $\E\sup_{f\in F} G_f =\E \|X\|_2 \sim\sqrt{d}$. 
In fact, it is standard to verify that $\|X\|_2$ concentrates sharply around $\sqrt{d}$ (for our purposes it is enough that $c_1 \sqrt d \leq \|X\|_2 \leq c_2\sqrt d$ holds with probability at least $1-2\exp(-c_3d)$).
Moreover,  for every $z \in S^{d-1}$, $\E |\inr{X,z}|^{p} =\E|g_1|^{p} \leq (c_4\sqrt p)^p$,  implying that ${\rm diam}(u(F)) \sim c_5(p)$. 
Therefore, if an estimate of the form \eqref{eq:subgaussian-est-2} were true, we would have that for $d = \beta N$ and for typical realizations of $(X_i)_{i=1}^N$,
$$
\sup_{z \in S^{d-1}} \left| \frac{1}{N}\sum_{i=1}^N |\inr{X_i,z}|^{p}- \E |\inr{X,z}|^{p} \right| 
\leq c_6(p)\sqrt{\frac{d}{N}} 
=c_6(p) \sqrt{\beta}.
$$
However,  with probability at least $1-2\exp(-c_3d)$, 
\begin{align}
\nonumber
\sup_{z \in S^{d-1}} \left| \frac{1}{N}\sum_{i=1}^N |\inr{X_i,z}|^{p}- \E |\inr{X,z}|^{p} \right| 
&\geq \sup_{z \in S^{d-1}} \frac{|\inr{X_1,z}|^{p} }{N}   - (c_4\sqrt p)^p \\
\nonumber
&= \frac{\|X_1\|_{2}^{p}}{N}-  c(p) 
 \geq  \frac{(c_1\sqrt d)^p}{N}  - c(p) \\
&=  c_1^p\beta^{p/2} \cdot N^{p/2-1} -  c(p).
\nonumber
\end{align}

Similar examples show that even for $p=2$, if $F$ satisfies a slightly weaker norm equivalence than \eqref{eq:subgaussian}, the empirical mean functional $\frac{1}{N}\sum_{i=1}^N f^2(X_i)$ does not satisfy \eqref{eq:nice-perform}. 
The reason behind both facts is that the empirical mean performs well only in `light-tailed' situations, and if the function $u(t)$ grows too quickly (and here, slightly faster than quadratic is already `too quickly' even if the starting point is Gaussian), the problem is more heavy-tailed than the empirical mean can handle. Since our main interest is dealing with classes that consist of functions that might be truly heavy-tailed --- say, with a tail behaviour $~\sim 1/t^4$, the empirical mean is not a viable option as far as the wildly optimistic conjecture goes.

What is equally significant is that there were no indications that \eqref{eq:nice-perform} had any chance of being true. The alternative uniform mean estimators that have been considered over the years never came close to exhibiting a subgaussian error for an arbitrary class of functions (see, e.g.\ \cite{lugosi2019near} or the survey \cite{lugosi2019mean} and the references therein).

\begin{tcolorbox}
Despite all that, the main result of this note is an affirmative answer to Question \ref{qu:main-intro}: we introduce a uniform mean estimator that attains a subgaussian error even in heavy-tailed scenarios. 
\end{tcolorbox}

It is important to emphasize at this point that the feasibility of the the estimator designed in what follows is a question of independent interest.
Nevertheless, the mere existence of such an estimator is both surprising and significant. 

We return to this issue again in Remark \ref{rem:computation} and in Section~\ref{sec:computation}.

\subsection{Assumptions and the main result}

Before formulating the main result, let us present the assumptions we need. 

\begin{assumption} 
\label{ass:iso-distance-orace}
One is given  $\kappa >0$ and a functional ${\mathbbm \rho}$ which satisfies that for every $f,h \in F \cup \{0\}$,
$$
\frac{1}{\kappa}\|f-h\|_{L_2} \leq {\mathbbm \rho}(f,h) \leq \kappa \|f-h\|_{L_2}.
$$
\end{assumption}

Thanks to Assumption \ref{ass:iso-distance-orace}, there is some additional, rather crude, apriori information on distances between functions in $F$, as well as on their $L_2$ norms. Here, the constant $\kappa$ should be viewed as a large number --- much bigger than $1$. As a result, the functional ${\mathbbm \rho}$ is likely to distort the true distances and norms significantly. 

\begin{remark}
It is important to note that there are problems in which the $L_2$ structure of $F$ is known (as is the case in the example presented in Section \ref{sec:application.Lp}). When that happens, Assumption \ref{ass:iso-distance-orace} it trivially satisfiesd.
\end{remark}

The second assumption we require is a weak norm equivalence satisfied by $F \cup\{0\}$, and a rather minor assumption on the function $u$: 

\begin{assumption} \label{ass:other}
The class $F$ is centrally symmetric; each function $f\in F$ has mean zero; and there is a constant $L$ for which, for every $f,h \in F \cup \{0\}$, 
$$
\|f-h\|_{L_4} \leq L \|f-h\|_{L_2}.
$$
Also, $u(0)=0$, there is an increasing function $v:\R_+ \to \R_+$ that satisfies for every $s,t \in \R$,
$$
|u(s)-u(t)| \leq v(|s|+|t|) \cdot |s-t|,
$$
and $\sup_{f \in F } \E v^4(2|f|) < \infty$.  
\end{assumption}

To put Assumption \ref{ass:other} in some context, the norm equivalence condition is minimal: it is still possible that functions in $F$ do not have a finite $4+\eps$ moments, and in particular they can be heavy-tailed. 
The central symmetry of $F$ is only assumed to ease notation---in the general case one may replace $F$ by $F \cup -F$.

The second part of Assumption \ref{ass:other} prohibits $v$ from growing `too quickly' relative to the tail behaviour of elements in $F$. For example, if $u(t)=|t|^p$, one may set $v(t) = p |t|^{p-1}$, and the integrability assumption holds if $F $ is a bounded subset of $L_{4(p-1)}$. That is not significantly more restrictive than the bare-minimum needed to ensure that Question \ref{qu:main-intro} makes sense, namely that $\sup_{f \in F} \E |u(f)| =\sup_{f \in F} \E |f|^p < \infty$.

\begin{remark}
Clearly, if $u$ is smooth and convex (as is the case in most interesting applications), one can choose $v=u^\prime$. The integrability assumption on $v$ is there to balance  the tail behaviour of functions in $F$ and the speed at which $u^\prime$ grows. 
\end{remark}

From here on, set 
$$
R(F) = \sup_{f\in F } \left(\E v^4(2|f|)\right)^{1/4}, \ \ \ \ \ d_F = \sup_{f \in F} \|f\|_{L_2},
$$
and let
$$
D^*(F) = \left(\frac{\E \sup_{f \in F} G_f}{d_F}\right)^2 
$$
be the \emph{critical dimension} of $F$.

\begin{remark}
The critical dimension is a natural object. It plays an instrumental role in Gaussian concentration \cite{artstein2015asymptotic,pisier1986probabilistic}; in the Dvoretzky-Milman theorem \cite{artstein2015asymptotic,milman1971new,pisier1986probabilistic}; and in many other situations. 

Roughly put, if the sample size $N$ is smaller than the critical dimension there simply isn't enough `mixing' to overcome the richness of $F$; as a result, the given random data does not suffice for accurate mean estimation in $F$.
\end{remark}

The main result of this article is as follows:

\begin{tcolorbox}
\begin{theorem} \label{thm:main}
If Assumptions \ref{ass:iso-distance-orace} and \ref{ass:other} are satisfied, there is an absolute constant $c_1$ and constants $c_2,c_3$ that  depend on $\kappa$ and $L$ for which the following holds.

For every  $\delta>\exp(-c_1N)$, there is a procedure $\Psi_\delta$ which satisfies that with probability at least $1-\delta$, 
\begin{align} \label{eq:main-2} 
\sup_{f\in F}\left| \Psi_\delta(X_1,...,X_N,f) - \E u(f) \right| 
 \leq c_2 R(F) \cdot  \left(\frac{\E \sup_{f \in F} G_f}{\sqrt{N}} + d_F \sqrt{\frac{\log(1/\delta)}{N}} \right).
\end{align}

In particular, with probability at least $1-\exp(-c_1\min\{D^*(F),N\})$, 
\begin{equation} \label{eq:main-1}
 \sup_{f\in F} \left| \Psi_\delta(X_1,...,X_N,f) - \E u(f) \right| 
\leq c_3 R(F) \cdot \frac{\E \sup_{f \in F} G_f}{\sqrt{N}}.
\end{equation}
\end{theorem}
\end{tcolorbox}

\begin{example}
Let us return to the case $u(t)=|t|^2$.  A straightforward application of Theorem \ref{thm:main} shows that under the $L_4-L_2$ norm equivalence, with probability at least $1-\exp(-c_1\min\{D^*(F),N\})$,
$$
\sup_{f \in F} \left| \Psi(X_1,...,X_N,f) - \E f^2 \right| 
\leq c_2(L,\kappa)  \cdot d_F \frac{\E \sup_{f \in F} G_f}{\sqrt{N}}.
$$
This coincides with the optimal  subgaussian error estimate, but holds for a class of heavy-tailed functions.
\end{example}

\begin{remark}
\label{rem:computation} 
It is important to stress that the procedure $\Psi_\delta$ introduced as part of the proof of Theorem~\ref{thm:main} is not computationally friendly. 
Its definition requires, as input, an almost optimal admissible sequence for $\gamma_2(F, \|\cdot\|_{L_2})$ (see Section~\ref{sec:generic.chaining} for the definition and more details), something that can be done---in theory---using Assumption~\ref{ass:iso-distance-orace}.
The construction of an optimal admissible sequence is a deterministic question and depends solely on the geometry of the space $(F, \rho)$. 
As such, Theorem \ref{thm:main} reveals a surprising phenomenon: that uniform mean estimation can be decoupled into two distinct challenges. 
The first is a deterministic one: understanding the geometry of the space $(F, \rho)$ and constructing an appropriate admissible sequence. 
This aspect remains open  in general, but is tractable in many concrete cases. 
The second is the statistical challenge: given access to such an admissible sequence, to construct a mean estimator with strong performance guarantees.
The focus of this article is on solving the latter problem.

Having said that, as we explain in Section~\ref{sec:computation}, constructions  of almost optimal admissible sequences are known for many important and relevant examples. Moreover, slightly sub-optimal admissible sequences are much easier to construct and result in a  marginal deterioration in the performance of the resulting estimator.
\end{remark}

Taking into account all the effort that has been invested over the years in the study of uniform mean estimation procedures  that were not only suboptimal but also valid only in rather specific instances of Question \ref{qu:main-intro}, it would be natural to expect that the proof of Theorem \ref{thm:main} should be hard. 
In fact, the proof is rather simple once the right mechanism is introduced: a combination of an optimal mean estimation procedure for a real-valued random variable with Talagrand's generic chaining mechanism. 

\vspace{0.5em}

There are numerous applications to Theorem \ref{thm:main}, but we will focus on two of them: approximating the $L_p$ structure endowed on $\R^d$ by an isotropic log-concave measure, and covariance estimation when the given data is heavy tailed and corrupted. These applications are presented in Sections \ref{sec:application.Lp} and  \ref{sec:corruption}.
We also note that related ideas based on modifications of Talagrand's generic chaining have recently been applied in \cite{bartl2025we} to obtain Gaussian-type error bounds for statistical learning problems.

\vspace{1em}

\noindent
\emph{Notation:} Throughout, $c, c_0, c_1$ denote strictly positive absolute constants that may change their value at each appearance.
If a constant $c$ depends on a parameter $\alpha$, that is denoted by $c = c(\alpha)$. 
We write $x\lesssim y$ if there is an absolute constant $c$ for which $x\leq cy$ and $x\sim y$ if both $x \lesssim y$ and $y \lesssim x$.
Moreover, $x\lesssim_\alpha y$  means  $x\leq c(\alpha) y$ and similarly for $\sim_\alpha$.

\begin{remark}
To avoid digressing into well-understood technical issues (such as measurability), we shall assume that $F$ contains a countable subset that can be used to approximate pointwise every function in $F$.
\end{remark}

\section{Preliminaries} \label{sec:pre}

Our starting point is a classical question in Statistics: whether the mean of a random variable can be estimated from random data. This question has been completely resolved, and there are various procedures that achieve the optimal behaviour:

\begin{theorem} \label{thm:mean-estimation}
There are absolute constants $c_1$ and $c_2$ for which the following holds. For every $\exp(-c_1N) \leq \delta \leq 1/2$ there is a mapping $\psi_\delta:\R^N \to \R$ which satisfies that for any random variable $X$ with finite mean and variance, with probability at least $1-\delta$, 
\begin{equation} \label{eq:mean-estimation}
\left|\psi_\delta(X_1,...,X_N) - \E X \right|
 \leq c_2 \sigma_X \sqrt{\frac{\log(1/\delta)}{N}}, 
\end{equation}
where $\sigma_X^2$ is the variance of $X$. 

\end{theorem}

\begin{remark}
It is standard to show that \eqref{eq:mean-estimation} is the best one can hope for even if the random variable is known to be Gaussian. 
\end{remark}

The first procedure satisfying this subgaussian error estimate was the \emph{Median of Means},  introduced  by Nemirovsky and Yudin \cite{nemirovskij1983problem} (see also Alon, Matias, and Szegedy \cite{alon1999space} and Jerrum, Valiant,  and Vazirani \cite{jerrum1886random}).
 We provide the short proof of this fact in Appendix \ref{sec:MOM.1d}, and refer to \cite{lugosi2019mean} for a recent survey on mean estimation procedures in $\R$ and in $\R^d$.

We will use optimal mean estimation procedures $\psi_\delta$ as a `black box'. The key in the proof of Theorem \ref{thm:main} is aggregating those estimators in the right way, and as it happens that aggregation is dictated by Talagrand's generic chaining mechanism. 

\subsection{Generic chaining}
\label{sec:generic.chaining}

Generic chaining was introduced as a way of controlling the behaviour of the supremum of a stochastic process $\{Z_t : t \in T\}$, where each $Z_t$ is a mean-zero random variable. The intuitive idea behind generic chaining is to consider approximating sets $T_s \subset T$ that become finer as $s$ grows, and set $\pi_s : T \to T_s$ to be suitable metric projection functions. Writing
$$
Z_t -Z_{\pi_0 t}= \sum_{s \geq 0} \left(Z_{\pi_{s+1}t} - Z_{\pi_s t}\right),
$$
and assuming that $|T_0|=1$, one may control $\E\sup_{t \in T} Z_t = \E \sup_{t \in T}(Z_t-Z_{t_0})$ by ensuring that with high probability each `link' $Z_{\pi_{s+1}t} - Z_{\pi_s t}$ is not `too big'. Combining such individual estimates with the fact that the number of possible links at the $s$-stage is at most $|T_s| \cdot |T_{s+1}|$ leads to the wanted estimate by using the union bound.  

Naturally, this idea is only the starting point of generic chaining theory. For one, $T$ is just an indexing set, and the right notion of `approximation' or `metric' is not obvious off-hand. Also, even if it is possible to establish an upper bound on $\E\sup_{t \in T} Z_t$ using generic chaining, it is not clear at all that the resulting bound is sharp. 
As it happens, when the random process $\{Z_t : t \in T\}$ is Gaussian, the chaining mechanism gives a complete characterization of $\E \sup_{t \in T} Z_t$. 

Before we formulate that striking characterization we need to introduce some elements of the generic chaining theory; those will also be required for the proof of Theorem \ref{thm:main}.

For a comprehensive exposition on generic chaining we refer the reader to Talagrand's treasured book \cite{talagrand2022upper}.

\begin{definition} \label{def:admissible}
An admissible sequence of a set $T$ is a collection $(T_s)_{s \geq 0}$ of subsets of $T$, satisfying that  $|T_s| \leq 2^{2^{s}}$ and $|T_0|=1$. If $(T,d)$ is a metric space, let $\pi_s t$ be a nearest point to $t$ in $T_s$ (with respect to the metric $d$). Finally, let
$$
\gamma_2(T,d) = \inf_{(T_s)_{s\geq 0}} \sup_{t \in T} \sum_{s \geq 0} 2^{s/2} d(t,\pi_st),
$$
where the infimum is taken with respect to all admissible sequences of $T$. 
\end{definition}

\begin{theorem}
\label{thm:mm}
There are absolute constants $c_1$ and $c_2$ for which the following holds. Let $\{G_t : t \in T\}$ be a centred Gaussian process, and endow $T$ with the $L_2$ metric, i.e., set $d(t,t^\prime) = \|G_t - G_{t^\prime}\|_{L_2}$. Then
$$
c_1 \gamma_2(T,d) \leq \E \sup_{t \in T} G_t \leq c_2 \gamma_2(T,d). 
$$
\end{theorem}

The fact is that the expected supremum of Gaussian processes is a purely metric object---determined by the $\gamma_2$ functional of the indexing set endowed with the right metric---, is nothing short of remarkable. 
There are only a handful of processes other than the Gaussian that have similar metric characterizations, and those characterizations involve more than one metric and other functionals \cite{bednorz2014boundedness, talagrand1994supremum,talagrand2022upper}.

The lower bound in Theorem \ref{thm:mm} is Talagrand's \emph{majorizing measures theorem} \cite{talagrand1987regularity} (see also the presentation in \cite{talagrand2022upper}). It implies that for every indexing set $T$ there is an admissible sequence $(T_s)_{s \geq 0}$ for which 
$$
\sup_{t \in T} \sum_{s \geq 0} 2^{s/2} \|\Delta_s t\|_{L_2} \leq c \E \sup_{t \in T} G_t
$$
where we abuse notation and write $\|\Delta_s t\|_{L_2}=\| G_{\pi_{s+1} t} - G_{\pi_s t}\|_{L_2}$ instead of $d(\pi_{s+1}t,\pi_st)$. 

The construction of an admissible sequence calls for  a solution of an optimization problem that depends on the underlying metric structure---in this case, the $L_2$ one. At this point, let us return to Assumption \ref{ass:iso-distance-orace} --- the existence of an $L_2$ distance oracle. The only instance that oracle is used in the proof of Theorem \ref{thm:main} is to ensure that one has access to an admissible sequence of $F$ that satisfies 
\begin{equation} \label{eq:admissible-1}
\sup_{f \in F} \sum_{s \geq 0}  2^{s/2}\|\Delta_s f\|_{L_2} 
\leq c(\kappa) \E \sup_{f \in F}G_f.
\end{equation}
That is guaranteed because a solution to the optimization problem of finding an optimal admissible sequence for $\gamma_2(F,\rho)$ leads to an almost optimal admissible sequence for $\gamma_2(F,L_2)$. Indeed, for every admissible sequence, 
$$
\sum_{s \geq 0} 2^{s/2}\|\Delta_s f\|_{L_2} \sim_\kappa  \sum_{s \geq 0} 2^{s/2}{\mathbbm \rho}(\pi_{s+1}f,\pi_s f),
$$
implying that the admissible sequence constructed using ${\mathbbm \rho}$ satisfies \eqref{eq:admissible-1}. 

As a result, we will assume from here on that 
\begin{tcolorbox}
One has access to an admissible sequence of $F$ for which  \eqref{eq:admissible-1} holds.
\end{tcolorbox}

As already noted in the introduction, the actual construction of an admissible sequence for $F$ satisfying \eqref{eq:admissible-1} is a nontrivial task. 
We refer the reader to Section~\ref{sec:computation} for a more detailed discussion on such constructions.

\vspace{0.5em}

The key to the proof of the upper bound in Theorem \ref{thm:mm} is a \emph{subgaussian increment condition}: if $\{Z_t : t \in T\}$ is a Gaussian process, then for every $t, t^\prime \in T$ and $u \geq 1$, 
\begin{equation} \label{eq:subgaussian-inc}
\PP \left( \left|Z_t - Z_{t^\prime}\right| \geq u 2^{s/2} \|t-t^\prime\|_{L_2} \right) \leq 2\exp(-cu2^{s}).
\end{equation}
The fact that $\E \sup_{t \in T} Z_t \leq \gamma_2(T,d)$ follows by employing \eqref{eq:subgaussian-inc} for each increment $\Delta_s t = \pi_{s+1} t - \pi_s t$, and using that $|\{\pi_{s+1} t - \pi_s t : t \in T\}| \leq 2^{2^s} \cdot 2^{2^{s+1}} \leq 2^{2^{s+2}}$.

Unfortunately, useful as that may be in the context of Gaussian processes, other processes do not enjoy a subgaussian increment condition like \eqref{eq:subgaussian-inc}. For example, the empirical mean fails miserably in that respect, unless the class in question is subgaussian. What saves the day are optimal mean estimation procedures: thanks to Theorem \ref{thm:mean-estimation} such procedures do satisfy a subgaussian increment condition. That, and the access to an almost optimal admissible sequence is enough for the proof of Theorem \ref{thm:main}.

\section{Proof of Theorem \ref{thm:main}}
Let $(F_s)_{s \geq 0}$ be an admissible sequence as in \eqref{eq:admissible-1}. 
Set $s_1$ to satisfy $2^{s_1} = c_1 N$ for a well-chosen absolute constant $c_1$, and let $s_0<s_1$ to be named in what follows.

For $s_0\leq s < s_1$, set $\delta_s = \exp(-2^{s+4})$ and let $\psi_{\delta_s}$ be a mean estimation procedure that satisfies \eqref{eq:mean-estimation} with the confidence parameter $\delta_s$; such a procedure exists for every $s_0 \leq s < s_1$ thanks to the choice $s_1$. Hence, for any random variable $Z$ and independent copies $Z_1,\dots,Z_N$ of $Z$, with probability at least $1-\exp(-2^{s+4})$, 
\begin{align}
\label{eq:modification.for.corruption}
\left|\psi_{\delta_s} (Z_1,...,Z_N) - \E Z \right| \leq c_2 \frac{2^{s/2}}{\sqrt{N}} \sigma_Z.
\end{align}

Note that for every $f \in F$, we have that 
\begin{align}
\label{eq:split.H(f)}
u(f) = \left(u(f)-u(\pi_{s_1} f)\right) + \sum_{s=s_0}^{s_1-1} \left(u(\pi_{s+1} f) - u(\pi_s f)\right) + u(\pi_{s_0}f).  
\end{align}
With that in mind, set
\begin{align*}
\Psi_s(f) &= \psi_{\delta_s} \left( \big(u(\pi_{s+1} f(X_i)) - u(\pi_s f(X_i))\big)_{i=1}^N\right),  \\
\Psi^\prime_{s_0}(f)& = \psi_{\delta_{s_0}}\left( \big( u(\pi_{s_0}f(X_i)) \big)_{i=1}^N \right),
\end{align*}
and define 
$$
\Psi(f) = \sum_{s=s_0}^{s_1-1} \Psi_s (f) + \Psi^\prime_{s_0}(f).
$$

To show that $\Psi$ performs with the promised accuracy and confidence, let 
$$
H_s = \left\{u(\pi_{s+1} f) - u(\pi_s f) : f \in F \right\},
$$ 
and observe that $|H_s| \leq 2^{2^{s+2}}$. 
By \eqref{eq:modification.for.corruption}, for every $h \in H_s$, with probability at least $1-\exp(-2^{s+4})$, 
\begin{equation} \label{eq:in-proof-1}
\left|\psi_{\delta_s}(h(X_1),\dots,h(X_N)) - \E h \right| 
\leq c_2 \frac{ 2^{s/2}}{\sqrt{N}} \cdot \|h\|_{L_2}.
\end{equation}
Using \eqref{eq:in-proof-1} for every $h \in H_s$, combined with the fact that $|H_s| \leq 2^{2^{s+2}}$, it follows from the union bound that with probability at least $1-\exp(-2^{s+3})$, for every $f \in F$,
\begin{align}
\label{eq:for.corruption}
\left| \Psi_s(f) - \E \left(u(\pi_{s+1} f) - u(\pi_s f)\right) \right| \leq c_2 \frac{ 2^{s/2} }{\sqrt{N}} \cdot \|u(\pi_{s+1} f) - u(\pi_s f)\|_{L_2}. 
\end{align}
Similarly, with probability at least $1-\exp(-2^{s_0+3})$, for every $f \in F$,
\begin{align}
\label{eq:for.corruption.2} 
\left|\Psi^\prime_{s_0}(f)- \E u(\pi_{s_0} f) \right|
\leq c_2\frac{ 2^{s_0/2}}{\sqrt{N}} \cdot  \|u(\pi_{s_0}f) \|_{L_2}. 
\end{align}
Therefore, by the  union bound over $s_0\leq s<s_1$ and \eqref{eq:split.H(f)}, we have that with probability at least $1-\exp(-c_3 2^{s_0})$, for every $f \in F$,
\begin{align*}
& \left| \Psi(f) - \E u(f)\right| 
\\
&\leq  \left|\E u(f) - \E \pi_{s_1} u(f)\right| + \sum_{s=s_0}^{s_1-1}  \left| \Psi_s(f)- \E \bigl(u(\pi_{s+1} f) - u(\pi_s f)\bigr)  \right|
+ \left|\Psi^\prime_{s_0}(f)- \E u(\pi_{s_0} f) \right|
\\
&\leq  \left|\E u(f) - \E \pi_{s_1} u(f)\right| + \frac{c_2}{\sqrt{N}} \left( \sum_{s=s_0}^{s_1-1} 2^{s/2} \|u(\pi_{s+1} f) - u(\pi_s f)\|_{L_2} + 2^{s_0/2}\|u(\pi_{s_0}f)\|_{L_2}\right)
\\
&=  (1)+(2).
\end{align*}
All that remains is to derive uniform bounds on $(1)$ and $(2)$. 

\vspace{0.5em}

To estimate $(1)$, note that by  Assumption \ref{ass:other},
\begin{align*}
(1) &\leq  \E |u(f)-u(\pi_{s_1}f)| \\
&\leq  \E |f-\pi_{s_1} f| \cdot |v(|f|+|\pi_{s_1}f|) | 
\leq  \|f-\pi_{s_1}f\|_{L_2} \cdot \|v(|f|+|\pi_{s_1}f|)\|_{L_2}.
\end{align*}
Moreover, since  $v$ is increasing and nonnegative,
\[v(|f|+|\pi_{s_1}f|) 
\leq \max\{  v(2|f|), v(2|\pi_{s_1}f|)\}
\leq  v(2|f|) + v(2|\pi_{s_1}f|)\]
pointwise.
Thus, recalling that $R(F)=\sup_{f\in F} (\E v^4(2|f|))^{1/4}$,
\begin{align}
\label{eq:estim.v}
\|v(|f|+|\pi_{s_1}f|)\|_{L_2}
\leq \|v(2|f|)\|_{L_2} + \|v(2|\pi_{s_1}f|)\|_{L_2}
\leq 2 R(F).
\end{align}

Also, since $2^{s_1} =c_0 N$, it is evident that
\begin{align*}
\|f-\pi_{s_1}f\|_{L_2} 
&\leq  \frac{1}{2^{s_1/2}} \sum_{s \geq s_1} 2^{s_1/2}\|\Delta_s f\|_{L_2}  \\
&\leq  \frac{1}{\sqrt{c_0 N}} \sum_{s \geq s_1} 2^{s/2}\|\Delta_s f\|_{L_2}
\leq \frac{c_4(\kappa)}{\sqrt{N}} \E \sup_{f\in F} G_f,
\end{align*}
where the last inequality follows from \eqref{eq:admissible-1}.
Thus 
\[
(1) \leq 2c_4(\kappa) \frac{R(F)}{\sqrt{N}} \cdot \E \sup_{f\in F} G_f.
\]

Turning to $(2)$, by Assumption \ref{ass:other},
\begin{align*}
\|u(\pi_{s+1} f) - u(\pi_s f)\|_{L_2} \leq & \|\Delta_s f \cdot v(|\pi_sf|+|\pi_{s+1}f|) \|_{L_2} 
\\
\leq & \|\Delta_s f\|_{L_4} \cdot 2 R(F)  
\leq L \|\Delta_s f\|_{L_2} \cdot 2 R(F),
\end{align*}
where we used the same argument as in \eqref{eq:estim.v} to obtain $\|v(|\pi_sf|+|\pi_{s+1}f|) \|_{L_4}\leq 2 R(F)$.

Using that $u(0)=0$, a similar argument for the term involving $s_0$ shows that pointwise
$$
|u(\pi_{s_0}f)|=|u(\pi_{s_0}f)-u(0)| \leq |\pi_{s_0}f| \cdot v(|\pi_{s_0}f|),
$$
and therefore,  
$$
\|u(\pi_{s_0}f)\|_{L_2} \leq L \|\pi_{s_0} f\|_{L_2} \cdot R(F).
$$

Thus we have that
\begin{align*}
(2)  
&\leq c_5(L) \frac{R(F)}{\sqrt{N}} \sup_{f \in F} \left(\sum_{s=s_0}^{s_1-1} 2^{s/2} \|\Delta_s f\|_{L_2} +2^{s_0/2} \sup_{f \in F} \|f\|_{L_2}\right)
\\
&\leq  c_6(\kappa,L) \frac{R(F)}{\sqrt{N}} \left(\E\sup_{f \in F}G_f + 2^{s_0/2}d_F\right).
\end{align*}

We conclude that with probability at least $1-\exp(-c_3 2^{s_0})$, for every $f \in F$, 
\begin{equation} \label{eq:proof-main-final}
\left| \Psi(f) - \E u(f)\right| 
\leq c_7(\kappa,L) \frac{R(F)}{\sqrt{N}} \left(\E\sup_{f \in F}G_f + 2^{s_0/2}d_F\right). 
\end{equation}

Now, the first part of the claim follows by setting  $2^{s_0} \sim \log(1/\delta)$; and the second part is evident by setting  $2^{s_0} \sim \min\{D^\ast(F),N\}$.
\qed


\section{A geometric application}
\label{sec:application.Lp}
Let $\mu$ be an isotropic, log-concave measure on $\R^d$. Thus, $\mu$ is centred, its covariance operator is the identity, and it has a density function that is log-concave. Understanding the behaviour of log-concave measures is of the utmost importance in Asymptotic Geometric Analysis because the volume measure on a convex, centrally symmetric set is log-concave (see, e.g., \cite{artstein2015asymptotic}). It is well understood that most of the `mass' of any isotropic log-concave measure lives in the same thin shell \cite{artstein2021asymptotic,chen2021almost,lee2024eldan}, and as a result, distinguishing between any two such measures is a nontrivial task. That  motivated the following question, posed by V.~Milman: 

\begin{question} \label{qu:milman}
Given two isotropic log-concave measures $\mu$ and $\nu$ and a sample $X_1,…,X_N$ drawn from one of them, how large should $N$ be to decide whether $X\sim\mu$ or $X\sim\nu$?
\end{question}

One possible way of distinguishing between probability measures is through the $L_p$ structures they endow on $\R^d$: if $X$ is distributed according to $\mu$, then identifying each $z \in \R^d$ with the linear functional $\inr{\cdot,z}$, the unit ball of $(\R^d,L_p(X))$ is given by 
$$
{\mathcal K}_p = \left\{ z \in \R^d : \E |\inr{X,z}|^p \leq 1\right\}. 
$$
The sets ${\mathcal K}_p$ characterize the distribution of $X$; therefore, if one would like to distinguish between two log-concave measures, it suffices to exhibit that their $L_p$ unit balls are different---at least for some $p$. 

\begin{remark}
As it happens, all the sets ${\mathcal K}_p$ endowed by an isotropic, log-concave measure are equivalent to the Euclidean unit ball $B_2^d$: on the one hand, by isotropicity, ${\mathcal K}_2 = B_2^d$; on the other,
$$
{\mathcal K}_p \subset {\mathcal K}_2 \subset (cp)^p {\mathcal K}_p.
$$ 
Indeed, one inclusion follows because of the natural hierarchy of the $L_p$ norms, while the other is an outcome of Borell's inequality (see, e.g., \cite{artstein2015asymptotic})---that for every $z \in \R^d$ and every $p \geq 2$, $\|\inr{X,z}\|_{L_p} \leq cp \|\inr{X,z}\|_{L_2}$. Of course, this information is rather coarse and does not allow one to pin-point the sets ${\mathcal K}_p$ accurately.
\end{remark}
 
Taking Question \ref{qu:milman} as a starting point, it is natural to look for \emph{membership oracles} of the sets ${\mathcal K}_p$. To be more accurate, let $T \subset S^{d-1}$ be a centrally symmetric set, and consider the cone ${\mathcal C}_T=\{\lambda t : t \in T, \ \lambda \geq 0 \}$. Every $t \in T$ corresponds to a direction in $\R^d$ and the hope is to identify the right $\lambda>0$ for which $\|\inr{X,\lambda t}\|_{L_p} =1$. To that end, the goal is to design $\{0,1\}$-valued functionals $\mathcal{M}_{T,p}$ that upon receiving $X_1,...,X_N$ that are iid copies of $X$, decide accurately, for every $z \in {\mathcal C}_T$, whether $z \in {\mathcal K}_p$ or not.   

\begin{tcolorbox} 
\begin{question} \label{qu:sampling}
Let $\alpha,\delta\in(0,1)$, $p\geq 2$, and consider $X_1,...,X_N$ that are independent copies of an isotropic, log-concave random vector $X$.
How large should $N$ be to have a membership oracle $\mathcal{M}_{T,p}$ for which, with probability at least $1-\delta$ with respect to the $N$-product measure, the set $\widehat{\mathcal{K}}_{T,p}=\{z \in {\mathcal C}_T :  \mathcal{M}_{T,p}(z)=1\}$ satisfies 
$$
(1-\alpha)\widehat{\mathcal{K}}_{T,p} 
 \subset  ( {\mathcal K}_p \cap {\mathcal C}_T )
 \subset  (1+\alpha) \widehat{\mathcal{K}}_{T,p}?
$$ 
\end{question}
\end{tcolorbox} 

As we explain in what follows, an equivalent formulation of Question \ref{qu:sampling} is finding a functional $\Phi$, constructed using $X_1,...,X_N$, for which $\Phi(z)$ is `close' to $\|\inr{X,z}\|^p_{L_p}$ on ${\mathcal C}_T$. In other words, for every $z \in {\mathcal C}_T$,
\begin{equation} \label{eq:iso-L_p}
(1-\eps)\E |\inr{X,z}|^p \leq \Phi(z) \leq (1+\eps)\E |\inr{X,z}|^p.
\end{equation}
That boils down to Theorem \ref{thm:main} for $u(t)=|t|^p$ and $F=\{\inr{\cdot,t} : t \in T\}$.

An estimate on \eqref{eq:iso-L_p} was established in \cite{guedon2007lp} for $T=S^{d-1}$ and 
$$
\Phi(z)=\frac{1}{N}\sum_{i=1}^N |\inr{X_i,z}|^p.
$$
It was shown that 
$$
N=c(p) \frac{ d^{p/2}\log(d)  }{\eps^2} 
$$
suffices for \eqref{eq:iso-L_p} to hold for every $z \in \R^d$. Moreover, for that choice of $\Phi$ the dependence on $d$ is optimal up to a logarithmic factor. In \cite{mendelson2021approximating} a different, less natural functional was introduced---again, only in the case $T=S^{d-1}$, leading to a linear dependence on $d$ (which clearly is the best one can hope for):

\begin{theorem}
 \label{thm:L-p-approximation}
There is a functional $\Phi$ for which, for every isotropic log-concave random vector $X$, with probability at least $1-2\exp(-c_1\eps^2 N)$, for every $z\in\R^d$,
$$
(1-\eps)\E |\inr{X,z}|^p \leq \Phi(X_1,...,X_N,z) \leq (1+\eps)\E |\inr{X,z}|^p
$$
provided that $N \geq c_2(p)\frac{ d\log(2/\eps)}{\eps^2} $. 
\end{theorem}

The condition on $N$ in Theorem \ref{thm:L-p-approximation} was subsequently improved to $N \geq c_2(p)\frac{ d}{\eps^2} $ in \cite{bartl2022structure}, but
the proof of Theorem \ref{thm:L-p-approximation} and the improvement in \cite{bartl2022structure} rely heavily on the fact that $T=S^{d-1}$ and do not extend beyond that case.
As such, those methods are not suitable for finding an answer to Question \ref{qu:sampling}, as that calls for dealing with an arbitrary subset of  the sphere.

In contrast, thanks to Theorem \ref{thm:main}, we have the following optimal estimate:

\begin{theorem}
For every $p\geq 2$ there are constants $c_1,c_2,c_3$ that depend only on $p$ for which the following holds. Let $T\subset S^{d-1}$ be a centrally symmetric set, consider the cone ${\mathcal C}_T=\{\lambda t : \ t \in T, \lambda >0\}$, set $\eps \leq c_1$ and let
\[ N\geq c_2 \left( \frac{ \E \sup_{t\in T} \inr{G,t} }{\eps}\right)^2.\]
Then there exists a random set $\widehat{\mathcal{K}}_{T,p}$ that is determined by $X_1,...,X_N$, for which, with probability at least $1-\exp(-c_3N \varepsilon^2 )$, 
\[ (1-\varepsilon) \widehat{\mathcal{K}}_{T,p}
\subset (\mathcal{K}_p \cap {\mathcal C}_T)
\subset (1+\varepsilon)  \widehat{\mathcal{K}}_{T,p} .\]
\end{theorem}

\begin{proof}
	Set $F=\left\{ \inr{\cdot,w} : w \in T\right\}$ and put $u(t)=|t|^p$. 
	In this case, using the isotropicity and log-concavity of $X$, 
\begin{align*}
R(F) \leq   2 \sup_{z \in  B_2^d} \left(\E |\inr{X,z}|^{4(p-1)}\right)^{1/4}
\leq  c_1(p). 
\end{align*}

Therefore, by Theorem \ref{thm:main}, with probability at least $1-\delta$, for every $w \in T$,
\begin{equation} \label{eq:L-p-est}
 \left| \Psi(X_1,...,X_N,w) - \E|\inr{X,w}|^p \right| 
\leq  c_3(p) \left( \frac{ \E \sup_{t\in T} \inr{G,t} }{ \sqrt N} + \sqrt\frac{\log(1/\delta)}{ N} \right)
\leq \varepsilon,
\end{equation}
where the last inequality holds if 
$$
N \geq \left( \frac{2c_3}{\varepsilon} \E \sup_{t\in T} \inr{G,t} \right)^2
$$ 
and for the choice $\delta=\exp(-c_4(p) N\varepsilon^2)$.

Now consider $z \in {\mathcal C}_T\setminus\{0\}$, let $w=z/\|z\|_2$ and put
$$
\Psi_1(X_1,...,X_N,z)=\|z\|_2^p \cdot \Psi(X_1,...,X_N,w).
$$
It is evident that for every $z \in {\mathcal C}_T\setminus\{0\}$,  $\frac{z}{\|z\|_2}\in T$ and 
\begin{equation} \label{eq:L-p-in-proof-1}
\|z\|_2^p \cdot \left| \Psi\left( X_1,...,X_N, \frac{z}{\|z\|_2}\right) - \E \left|\inr{X,\frac{z}{\|z\|_2}}\right|^p \right| \leq \|z\|_2^p \eps. 
\end{equation}
Recall that ${\mathcal K}_2 \sim_p {\mathcal K}_p$, and that ${\mathcal K}_2=B_2^d$. Therefore, $\|z\|_2^p \sim_p \E |\inr{X,z}|^p$, and \eqref{eq:L-p-in-proof-1} becomes 
\begin{equation} \label{eq:L-p-in-proof-2}
 \left| \Psi_1(X_1,...,X_N,z) - \E |\inr{X,z}|^p \right| \leq c_5(p)\eps \E |\inr{X,z}|^p.  
\end{equation}

Finally, set 
\[ 
\widehat{\mathcal{K}}_{T,p }
=\{ z\in {\mathcal C}_T :  \Psi_1(X_1,\dots,X_N,z) \leq 1 \},
\]
fix a realization for which \eqref{eq:L-p-est} holds, and let $z\in \mathcal{C}_T$.

It is straightforward to verify that if $z \in \widehat{\mathcal{K}}_{T,p }$ then 
$$
\E|\inr{X,z}|^p \leq \frac{1}{1-c_5\eps} \Psi_1(X_1,...,X_N,z)
\leq \frac{1}{1-c_5\eps},
$$
and if $z \not \in \widehat{\mathcal{K}}_{T,p }$ then 
$$
\E|\inr{X,z}|^p > \frac{1}{1+c_5\eps} \Psi_1(X_1,...,X_N,z)
> \frac{1}{1+c_5\eps} ,
$$
as claimed. 
\end{proof}

\section{Adversarial corruption}
\label{sec:corruption}

A natural question in high dimensional statistics is  designing uniform mean estimators that are accurate even if a part of the sample is corrupted, see e.g.\ \cite{abdalla2022covariance,catoni2012challenging,diakonikolas2019robust,lugosi2019risk}.
As it happens, the chaining  based argument that is used in the proof of Theorem \ref{thm:main} fits that setup nicely. 

Let $X_1,\dots,X_N$ be independent copies of $X$ and set  $Y_1,\dots,Y_N$ to be a corrupted sample with corruption level $\eta\in[0,1]$: an adversary can choose to change at most $\eta N$ of the points $X_1,...,X_N$ in any way they like. Of course, there is no information on which of the sample points have actually been corrupted, nor on the way in which they have been corrupted. 

\begin{theorem} \label{thm:main.corrpution}
Under Assumptions \ref{ass:iso-distance-orace} and \ref{ass:other}, there is an absolute constant $c_1$ and a constant $c_2=c_2(L,\kappa)$ for which the following holds.
If $\delta \geq \exp(-c_1N) $ and $\eta\in[0,1]$, there is a procedure $\Psi_{\delta,\eta}$ for which, with probability at least $1-\delta$, 
\begin{align} \label{eq:main-2.cor}
\begin{split}
& \sup_{f\in F}\left| \Psi_{\delta,\eta}(Y_1,...,Y_N,f) - \E u(f) \right|  \\
&\qquad \leq  c_2 R(F)\left(\frac{\E \sup_{f \in F} G_f}{\sqrt{N}} + d_F \sqrt{\frac{\log(1/\delta)}{N}} + d_F\sqrt{\eta} \right).
\end{split}
\end{align}
\end{theorem}

The dependence of \eqref{eq:main-2.cor}  on $\eta$ is known to be optimal unless one imposes stronger assumptions on the tail behaviour of functions $u(f)$.

The starting point in the proof of Theorem \ref{thm:main.corrpution} is finding an appropriate `subgaussian increment condition' that holds  in the corrupted setup. 

\begin{theorem} \label{thm:mean-corrupt}
There are absolute constants $c_1,c_2$ and $\eta_0<\frac{1}{2}$ and for every $\exp(-c_1N) \leq \delta \leq \frac{1}{2}$  and $\eta\in[0,\eta_0)$ there is a mean estimation procedure $\psi_{\delta,\eta}$ for which the following holds. 
If $X$ is a random variable and $Y_1,...,Y_N$ is a sample with at most $\eta N$ corrupted points, then with probability at least $1-\delta$,
\begin{equation} \label{eq:mean-corrupt}
\left| \psi_{\delta,\eta}(Y_1,...,Y_N) - \E X \right| 
\leq c_2\sigma_X \left(\sqrt{\frac{\log(1/\delta)}{N}} + \sqrt{\eta}\right).
\end{equation}
\end{theorem}

One can show that the Median of Means estimator satisfies \eqref{eq:mean-corrupt}, see Appendix \ref{sec:MOM.1d}.
Another construction can be found in  \cite{lugosi2021robust}.

Thanks to Theorem \ref{thm:mean-corrupt}, the proof of Theorem \ref{thm:main.corrpution} follows the very same path  used in the proof of Theorem \ref{thm:main}, and we will only sketch the argument.

\begin{proof}[Proof of Theorem \ref{thm:main.corrpution}]
We begin with the trivial case $\eta\in[c_0,1]$, where $c_0$ is a well chosen  absolute constant.
Set $\Psi_{\delta,\eta}=0$ and by using Assumption \ref{ass:other}, for every $f\in F$,
\[ |\Psi_{\delta,\eta}(f) - \E u(f)|
\leq \E | u(0) - u(f) | 
\leq \E |f| v(|f|)
\leq d_F R(F)
\leq \frac{1}{\sqrt{c_0}}  d_F R(F)  \sqrt{\eta}, \]
as required.

In case that $\eta<c_0$, let  $\delta_s=\exp(-2^{s+4})$, set $2^{s_0}\sim\max\{ \log(1/\delta), \eta N\}$, put $2^{s_1}\sim N$, and consider
\begin{align*}
\Psi_{\delta,\eta}(f) &= \sum_{s=s_0}^{s_1-1} \Psi_{s,\eta} (f) + \Psi^\prime_{s_0, \eta}(f)\\
&= \sum_{s=s_0}^{s_1-1} \psi_{\delta_s,\eta} \left( \left(u(\pi_{s+1} f(Y_i)) - u(\pi_s f(Y_i))\right)_{i=1}^N\right)+ \psi_{\delta_{s_0},\eta}\left( \left( u(\pi_{s_0}f(Y_i)) \right)_{i=1}^N \right).
\end{align*}
For every $s\geq s_0$, the error in \eqref{eq:mean-corrupt} scales just as in the uncorrupted case, because  
\[ 
\sqrt{\frac{\log(1/\delta_s)}{N}}  + \sqrt{\eta}
\sim \sqrt{\frac{2^s}{N}}  + \sqrt{\eta}
\sim \sqrt{\frac{2^s}{N}}. 
\]
Therefore, following the proof of Theorem \ref{thm:main} line-by-line, we have that with probability at least $1-\exp(-c_1 2^{s_0})$,
 \[  
|\Psi_{\delta,\eta}(f) - \E u(f)|
\leq c_2(L,\kappa) \frac{R(F)}{\sqrt{N}} \left(\E\sup_{f \in F}G_f + 2^{s_0/2}d_F\right),
\]
and the proof is completed by the choice of $s_0$.
\end{proof}

\subsection{Corrupted covariance estimation}
Even in the simplest scenario, where the class $F$ consists of linear functionals on $\R^d$, dealing with corrupted data is interesting for every reasonable choice of $u$ (e.g., $u(t)=|t|^p$), and an arbitrary $T \subset \R^d$. But to put Theorem \ref{thm:main.corrpution} in context, it is instructive to explore the case $u(t)=t^2$ and $T=S^{d-1}$---corresponding to \emph{corrupted covariance estimation} in $\R^d$. 

Let $X$ be an unknown, centred random vector in $\R^d$. 
 Assume that $X$ is square integrable and denote its covariance matrix by  $\Sigma_X$.
One receives as data a random sample consisting of $N$ independent copies of $X$ where at most $\eta N$ of them have been corrupted by an adversary. The goal is to find a symmetric, positive semi-definite matrix $\widehat{\Sigma}_{\delta,\eta}$ that approximates the covariance matrix of $X$: with probability at least $1-\delta$, 
\begin{align}
\label{eq:cov.estimaton.1}
\left\| \widehat{\Sigma}_{\delta,\eta} - \Sigma_X \right\|_{\rm op}
=\sup_{z,z^\prime \in S^{d-1}} \left|\inr{\widehat{\Sigma}_{\delta,\eta} \, z,z^\prime} - \inr{\Sigma_X z,z^\prime} \right| \leq \eps.
\end{align}

Since $\widehat{\Sigma}_{\delta,\eta}$ and $\Sigma_X$ are symmetric and positive semi-definite, one may restrict the supremum in \eqref{eq:cov.estimaton.1} to $z^\prime=z$, leading to the choice of $T=S^{d-1}$ and $u(t)=t^2$. 

\begin{remark}
To a certain extent, $T=S^{d-1}$ is a relatively easy choice in the context of Theorem \ref{thm:main.corrpution}. The sphere is both well-structured and rich, and the error caused by that richness can mask the effects of the random sample---making a sharp estimate simpler to derive.
\end{remark}
Thanks to the given functional $\mathbbm{\rho}$, Assumption \ref{ass:iso-distance-orace} implies that one has some information on the covariance matrix $\Sigma_X$; that is not the case in standard covariance estimation problems. Still, even if $\mathbbm{\rho}$ is given using a positive semi-definite matrix $A$, all that Assumption \ref{ass:iso-distance-orace} yields is that for every $z \in \R^d$, 
\begin{equation} \label{eq:iso-corrupt}
\kappa^{-2} \inr{\Sigma_X z,z} \leq \inr{A z,z} \leq \kappa^2 \inr{\Sigma_X z,z}.
\end{equation}
In contrast, the goal in covariance estimation is to derive significantly sharper bounds on $\Sigma_X$, and off-hand it is not clear that the isomorphic estimate \eqref{eq:iso-corrupt} is any real help towards achieving that goal. Thanks to Theorem \ref{thm:main.corrpution}, we will show that the crude isomorphic estimate leads to an optimal one.  
 
\vskip0.3cm
Let us show how  a suitable uniform mean estimator (that need not be related to a positive semi-definite matrix) leads to an optimal choice of $ \widehat{\Sigma}_{\delta,\eta}$. 

First, suppose that there is a procedure $\Psi_{\delta,\eta}$ which satisfies that with probability $1-\delta$, 
\begin{equation} \label{eq:cov-est-0}
\sup_{z \in S^{d-1}} \left| \Psi_{\delta,\eta}(Y_1,...,Y_N,z) - \E\inr{X,z}^2 \right| \leq \eps,
\end{equation}
and set $\widehat{\Sigma}_{\delta,\eta}$ to be any positive semi-definite matrix for which
\begin{equation} \label{eq:cov-est-1}
\sup_{z \in S^{d-1}} \left|\inr{\widehat{\Sigma}_{\delta,\eta}\, z,z} - \Psi_{\delta,\eta}(Y_1,...,Y_N,z)  \right| \leq \eps.
\end{equation}
Thanks to \eqref{eq:cov-est-0} such a matrix exists: 
Clearly
\begin{align}
\label{eq:cov.rep}
\E \inr{X,z}^2 = \inr{ \Sigma_Xz,z} 
\end{align}
and hence \eqref{eq:cov-est-1} holds for $\widehat{\Sigma}_{\delta,\eta}=\Sigma_X$. 
At the same time, if $\widehat{\Sigma}_{\delta,\eta}$ satisfies \eqref{eq:cov-est-1}, it follows from \eqref{eq:cov-est-0} and \eqref{eq:cov.rep} that
\begin{align}
\label{eq:cov.triangle}
\left\| \widehat{\Sigma}_{\delta,\eta} - \Sigma_X \right\|_{\rm op}
= \sup_{z \in S^{d-1}} \left|\inr{\widehat{\Sigma}_{\delta,\eta}\, z,z} - \inr{\Sigma_X z,z} \right|
  \leq 2 \varepsilon,
 \end{align}
as required. 

\vspace{0.5em}

The optimal answer on covariance estimation---even if a fraction of the  sample points have been corrupted---, was given recently in \cite{abdalla2022covariance} (see also \cite{oliveira2022improved}).
To formulate it, let $\lambda_1$ be the largest singular value of $\Sigma_X$ and denote its trace by ${\rm Tr}(\Sigma_X)$.

\begin{theorem}
\label{thm:AZ}
Assume that $X$ has mean zero and satisfies $L_4-L_2$ norm equivalence with constant $L$, i.e.\ $\|\inr{X,z}\|_{L_4}\leq L \|\inr{X,z}\|_{L_2}$ for every $z\in \R^d$.
Let $\delta\in(0,1)$ and $\eta\in[0,1]$. 
If  $N \geq c_1 ({\rm Tr}(\Sigma_X) + \log(1/\delta))$, there exists procedure $\widehat{\Sigma}_{\delta,\eta}$ for which, with probability at least $1-\delta$,
\begin{align}
\label{eq:AZ}
 \left\| \widehat{\Sigma}_{\delta,\eta} - \Sigma_X \right\|_{\rm op}
\leq c_2(L)  \lambda_1 \left( \sqrt \frac{{\rm Tr}(\Sigma_X)}{ \lambda_1 N} + \sqrt{\frac{\log(1/\delta)}{N}} + \sqrt{\eta} \right).
\end{align}
\end{theorem}

\begin{remark}
Theorem \ref{thm:AZ} is formulated here in the `hardest' regime, when $X$ satisfies only $L_4-L_2$ norm equivalence.
It is shown in \cite{abdalla2022covariance} that the error in \eqref{eq:AZ} is the best one can hope for in that case, but the dependence on $\eta$ can be improved if $X$ satisfies stronger norm equivalence assumptions, for example, the $L_p-L_4$ norm equivalence for some $p>4$.
\end{remark}

Theorem \ref{thm:main.corrpution} can be used to recover the optimal estimate from Theorem \ref{thm:AZ} using a much simpler argument, though with the caveat of  Assumption \ref{ass:iso-distance-orace}. 

\vskip0.3cm

Set $u(t)=t^2$ and then $v(t) = t$ is a valid choice.
Let
$F=\{\inr{\cdot,z} : z \in S^{d-1}\}$, and note that, by \eqref{eq:cov.rep},
\[
d_F =\sup_{z \in S^{d-1}} \|\inr{X,z}\|_{L_2} = \sqrt{\lambda_1}.
\] 
By the $L_4-L_2$ norm equivalence, 
\begin{align*}
R(F) 
=  2\sup_{z \in S^{d-1}} \left(\E \inr{X,z}^4\right)^{1/4} 
\leq 2 L\sqrt{\lambda_1},
\end{align*}
and because $G$ has the same covariance as $X$, 
$$
\E\sup_{z \in S^{d-1}} \inr{G,z} \leq ( \E\|G\|_2^2)^{1/2} = (\E \|X\|_2^2 )^{1/2} = \sqrt{{\rm Tr}(\Sigma_X)}. 
$$

Therefore, it follows from Theorem \ref{thm:main.corrpution} that there exists $\Psi_{\delta,\eta}$ that satisfies, with probability at least $1-\delta$,  
\begin{align*}
 &\sup_{z\in S^{d-1}}\left| \Psi_{\delta,\eta}(Y_1,...,Y_N,\inr{\cdot, z}) - \E \inr{X,z}^2 \right| \\
&\qquad\leq c(L,\kappa) \sqrt{\lambda_1} \left( \sqrt \frac{ {\rm Tr}(\Sigma_X)}{N} +  \sqrt{\lambda_1} \sqrt{\frac{\log(1/\delta)}{N}} + \sqrt{\lambda_1} \sqrt{\eta} \right)
=(\ast).
\end{align*}
Set $\widehat{\Sigma}_{\delta,\eta}$ to be any symmetric positive semi-definite matrix for which 
\[
\sup_{z\in S^{d-1}}\left| \inr{ \widehat{\Sigma}_{\delta,\eta} \,z, z} - \Psi_{\delta,\eta}(Y_1,...,Y_N,\inr{\cdot, z})  \right| \leq (\ast),
\]
and thus $\|\widehat{\Sigma}_{\delta,\eta} - \Sigma_X\|_{\rm op}\leq 2(\ast)$ by \eqref{eq:cov.triangle}.
\qed

\section{Computation considerations}
\label{sec:computation}

The mean estimation procedures considered in this article rely on the existence of an almost optimal admissible sequence $(F_s)_{s \geq 0}$ for $\gamma_2(F, \|\cdot\|_{L_2})$---that is, a sequence satisfying
\[
\sup_{f \in F} \sum_{s \geq 0} 2^{s/2} \|f - \pi_s f\|_{L_2}
\leq c(\kappa) \cdot \gamma_2(F, \|\cdot\|_{L_2}).
\]
To construct such a sequence, it is assumed that the statistician has access to a symmetric function $\rho$ that is equivalent to the $L_2$ norm; see Assumption~\ref{ass:iso-distance-orace}.
While this assumption guarantees the theoretical existence of such an admissible sequence, its \emph{practical construction} is, in general, a highly nontrivial task.

Nevertheless, as we explain in this section, there are various concrete and relevant examples for which explicit constructions are either known or feasible. Moreover, even when an almost optimal sequence for $\gamma_2(F, \|\cdot\|_{L_2})$ is not available, one can often construct a (slightly) sub-optimal admissible sequence that still leads to mean estimation procedures with a nearly optimal error bound.

To illustrate this fact, consider the case of a class of linear functionals on $\R^d$, that is, each $f\in F$ is of the form  $f(X) = \inr{X, t}$ for $t \in T\subset \R^d$, and $X \in \mathbb{R}^d$ is an isotropic random vector. 
In this setting, we have
\[
\gamma_2(F, \|\cdot\|_{L_2}) = \gamma_2(T, \|\cdot\|_2),
\]
and  an almost optimal admissible sequence for $\gamma_2(T, \|\cdot\|_2)$ is known for various sets $T$. 
These include $\ell_p^d$-balls (for $p \in [1, \infty]$), weighted $\ell_p^d$-balls, localized versions thereof (e.g., intersections with centred Euclidean balls of arbitrary radius), ellipsoids, and even certain more general convex bodies \cite{talagrand2022upper,van2018chaining,van2018chainingII}.

In particular, a suitable admissible sequence is always known in the  context of Theorem~\ref{thm:AZ},  because
\[
\gamma_2(B_2, \|\cdot\|_{L_2}) = \gamma_2\left(\Sigma_X^{1/2} B_2, \|\cdot\|_2\right),
\]
and the fact that $\Sigma_X^{1/2} B_2$ is an ellipsoid, that, at the `isomorphic level' is par of the information given to the statistician. 

Beyond such cases where almost optimal admissible sequences are explicitly known, one may resort to Dudley's entropy integral (${\rm D}_2$) to construct admissible sequences. 
To that end, first note that
\[
\gamma_2(F, d) 
\leq \sum_{s \geq 0} 2^{s/2} e_s(F, d)
\sim {\rm D}_2(F, d) = \int_0^\infty \sqrt{ \log \mathcal{N}(F, d, \varepsilon) }\, d\varepsilon,
\]
 where $e_s(F, d)$ is the smallest achievable covering error of $F$ with $2^{2^s}$ points, and $\mathcal{N}(F, d, \varepsilon)$ is the minimal number of $\varepsilon$-balls needed to cover $F$ with respect to the metric $d$.
 
The problem of constructing a sequence $(F_s)_{s\geq 0}$ that nearly achieves the optimal $e_s(F,d)$ is significantly simpler than constructing an almost optimal admissible sequence.
For example, there are standard constructions for all the aforementioned examples (e.g., $\ell_p^d$-balls and ellipsoids when vied as classes of linear functionals), as well as more complex, infinite-dimensional examples---such as classes of Lipschitz functions,  functions with Besov-type regularity, or parametrized classes of functions \cite{vaart1996weak}.

If our mean estimation procedure is defined using an admissible sequence $(F_s)_{s \geq 0}$ that nearly achieves the values $e_s(F, \|\cdot\|_{L_2})$---i.e.\ satisfies that
\[ \sum_{s\geq 0} \sup_{f\in F}  2^{s/2} \|f-\pi_sf \|_{L_2} \leq c(\kappa) \cdot {\rm D}_2(F,\|\cdot\|_{L_2}),\] 
---then Theorem~\ref{thm:main} remains valid with  $\E \sup_{f \in F} G_f \sim \gamma_2(F, \|\cdot\|_{L_2})$ replaced by ${\rm D}_2(F, \|\cdot\|_{L_2})$. 
Specifically, in the setting of Theorem~\ref{thm:main} and using its notation, for every $\delta > \exp(-c_1 N)$, there exists a procedure $\Psi_\delta'$ such that, with probability at least $1 - \delta$,
\[
\sup_{f \in F} \left| \Psi_\delta'(X_1, \dots, X_N, f) - \E u(f) \right|
\leq c_2 R(F) \cdot \left( \frac{{\rm D}_2(F, \|\cdot\|_{L_2})}{\sqrt{N}} + d_F \sqrt{ \frac{\log(1/\delta)}{N} } \right).
\]

Finally, while the bound $\gamma_2(F, \|\cdot\|_{L_2}) \lesssim {\rm D}_2(F, \|\cdot\|_{L_2})$ is not always tight, it is often sharp up to logarithmic factors. 
For instance, when $F$ is a class of linear functional on $\R^d$, it is straightforward to show that 
\[
{\rm D}_2(F, \|\cdot\|_{L_2}) \leq c \gamma_2(F, \|\cdot\|_{L_2}) \cdot \log(d).
\]
In such cases, with  probability at least $1 - \delta$,
\[
\sup_{f \in F} \left| \Psi_\delta'(X_1, \dots, X_N, f) - \E u(f) \right|
\leq c_3 R(F) \cdot \left( \frac{ \E \sup_{f \in F} G_f \cdot \log(d) }{\sqrt{N}} + d_F \sqrt{ \frac{\log(1/\delta)}{N} } \right),
\]
and the admissible sequence used in the definition of the procedure $\Psi_\delta'$ can be constructed.

\appendix

\section{The one-dimensional Median of Means}
\label{sec:MOM.1d}
Here we sketch the proofs that the Median of Means estimator satisfies the claims made in  Theorem \ref{thm:mean-estimation} and Theorem \ref{thm:mean-corrupt}, starting with the former.

Let $c_1$ be a suitable absolute constant.
Split the sample into $m = c_1 \log(\frac{1}{\delta})$ disjoint blocks of equal size  $n =\frac{N}{m}$ and let $\bar Y_j$ denote the empirical mean on block $j$ (assume without loss of generality that $n$ and $m$ are integers).
Define $\psi_{\delta}$ as the median of $(\bar Y_j)_{j=1}^m$.

Since $\bar Y_j$ is the mean of $n$ i.i.d.\ copies of $X$, Markov's inequality implies that
\[
\P\left(|\bar Y_j-\mathbb EX| \leq  4 \frac{ \sigma_X}{\sqrt n} \right) \geq \frac{15}{16},
\]
and, by  Hoeffding's inequality, with probability at least $1-\exp(-c_2m)=1-\delta$, 
\begin{align}
\label{eq:mom}
\left|\left\{ j\leq m :  |\bar Y_j-\mathbb EX| \leq 4 \frac{ \sigma_X}{\sqrt n}  \right\} \right| \geq 0.8m.
\end{align}
On that event,  $ |\psi_{\delta}-\mathbb EX| \leq 4 \frac{ \sigma_X}{\sqrt n}$, and the proof of Theorem \ref{thm:mean-estimation} follows from the definition of $n$. 

In the case of corruption (with level $\eta$), set $\delta'=\min\{\delta, \exp(-\frac{10\eta N}{c_1})\}$ and split the sample into $m' = c_1\log(\frac{1}{\delta'})= \max\{c_1\log(\frac{1}{\delta}), 10\eta N\}$ blocks, each of cardinality $n'=\frac{N}{m'}$.
Then there are no corrupted samples in least $0.9m'$ of the blocks, and the above arguments show with probability at least $1-\exp(-c_2m')\geq 1-\delta$, \eqref{eq:mom} holds with  $0.8m$ and $n$ replaced by $0.7m'$ and $n'$, respectively. 
Hence, denoting  the median of $(\bar Y_j)_{j=1}^{m'}$ by  $\psi_{\delta,\eta}$, it follows that with probability at least $1-\delta$, 
\[ |\psi_{\delta,\eta}-\mathbb EX|
 \leq 4 \frac{ \sigma_X}{\sqrt{n'}}
 =4\sigma_X\max\left\{ \sqrt\frac{c_1\log(1/\delta)}{N}, \sqrt{10\eta} \right\} ,  \]
proving Theorem \ref{thm:mean-corrupt}.

\vspace{1em}

\noindent
{\bf Acknowledgement:} This research was funded in whole or in part by the Austrian Science Fund (FWF) [doi: 10.55776/P34743 and 10.55776/ESP31], the Austrian National Bank [Jubil\"aumsfond, project 18983],  and a Presidential Young Professorship grant [`Robust statistical learning for complex data'].

\bibliographystyle{abbrv}

\end{document}